\documentclass[12pt]{amsart}
\usepackage{amssymb}  
\usepackage{latexsym} 
\usepackage[all]{xy}
\usepackage{url}
\usepackage{comment,multirow}
\usepackage{amsmath}
\usepackage{amsthm}

\newcommand{\bG}{\mathbf G}
\newcommand{\bH}{\mathbf H}
\newcommand{\Q}{\mathbb Q}
\newcommand{\PP}{\mathbb P}
\newcommand{\Z}{\mathbb Z}

\newcommand{\charic}{\operatorname{char}}
\newcommand{\disc}{\operatorname{disc}}
\newcommand{\Hbar}{{\overline{H}}}
\newcommand{\PGL}{\operatorname{PGL}}
\newcommand{\Jac}{\operatorname{Jac}}
\newcommand{\Pf}{\operatorname{Pf}}
\newcommand{\rank}{\operatorname{rank}}
\newcommand{\GL}{\operatorname{GL}}
\newcommand{\SL}{\operatorname{SL}}
\newcommand{\OK}{{\mathcal{O}_K}}
\newcommand{\adj}{{\operatorname{adj}}}
\newcommand{\diag}{\operatorname{Diag}}
\newcommand{\ra}{\longrightarrow}

\newtheorem{theorem}{Theorem}[section]

\newtheorem{lemma}[theorem]{Lemma}

\usepackage{setspace}

\theoremstyle{definition}
\newtheorem{definition}[theorem]{Definition}
\newtheorem{example}[theorem]{Example}
\newtheorem{remark}[theorem]{Remark}
\newtheorem{algorithm}[theorem]{Algorithm}

\begin{document}
\date{12th September 2023}
\title{Minimisation of 2-coverings of genus 2 Jacobians}

\author{Tom Fisher}
\address{University of Cambridge,
          DPMMS, Centre for Mathematical Sciences,
          Wilberforce Road, Cambridge CB3 0WB, UK}
\email{T.A.Fisher@dpmms.cam.ac.uk}

\author{Mengzhen Liu}
\address{Harvard University,
Department of Mathematics,
Science Center Room 325,
1~Oxford Street,
Cambridge, MA 02138, USA}
\email{aliu@math.harvard.edu}

\renewcommand{\baselinestretch}{1.1}
\renewcommand{\arraystretch}{1.3}

\renewcommand{\theenumi}{\roman{enumi}}

\begin{abstract}
  An important problem in computational arithmetic geometry is to find
  changes of coordinates to simplify a system of polynomial equations
  with rational coefficients. This is tackled by a combination of two
  techniques, called minimisation and reduction. We give an algorithm
  for minimising certain pairs of quadratic forms, subject to the 
  constraint that the first quadratic form is fixed.
  This has applications to 2-descent on the Jacobian of a genus $2$ curve.
\end{abstract}

\maketitle

\section{Introduction}

\subsection{Models for $2$-coverings}
\label{sec:models}

We work over a field $K$ with $\charic(K) \not= 2$.  Let $C$ be a
smooth curve of genus $2$ with equation
$y^2 = f(x) = f_6 x^6 + f_5 x^5 + \ldots + f_1 x + f_0$ where
$f \in K[x]$ is a polynomial of degree $6$.  We fix throughout the
polynomial
\begin{equation*}
G = z_{12}z_{34}-z_{13}z_{24}+z_{23}z_{14}.
\end{equation*}

The following two definitions are based on those
in~\cite[Section 2.4]{genus2ctp}.

\begin{definition}
  \label{def:model}
  A {\em model} (for a $2$-covering of the Jacobian of $C$) is a pair
  $(\lambda,H)$ where $\lambda \in K^\times$ and
  $H \in K[z_{12},z_{13},z_{23},z_{14},z_{24},z_{34}]$ is a quadratic
  form satisfying
  \[ \det(\lambda x \bG - \bH) = -\lambda^6 f_6^{-1} f(x) \]
  where $\bG$ and $\bH$ are the matrices of second partial derivatives
  of $G$ and $H$.
\end{definition}

We identify the space of column vectors of length $6$ and the space of
$4 \times 4$ alternating matrices via the map
\[ A : z = \begin{pmatrix} z_{12} \\ z_{13} \\ z_{23} \\ z_{14} \\ z_{24} \\ z_{34}
  \end{pmatrix}
  \mapsto \begin{pmatrix}
    0 & z_{12} & z_{13} & z_{14} \\
    & 0 & z_{23} & z_{24} \\
   \multicolumn{2}{c}{\multirow{2}{*}{\,$-$}}  & 0 & z_{34} \\
    & & & 0
  \end{pmatrix} \]
so that $G(z)$ is the Pfaffian of $A(z)$.  Then
each $4 \times 4$ matrix $P$ uniquely determines a $6 \times 6$ matrix
$\wedge^2 P$ such that
\[ P A(z) P^T = A( (\wedge^2 P) z) \] for all column vectors $z$. For
$F \in K[x_1, \ldots, x_n]$ and $M \in \GL_n(K)$ we write $F \circ M$
for the polynomial satisfying $(F \circ M)(x) = F(Mx)$ for all columns
vectors $x$.  The Pfaffian $\Pf(A)$ of an alternating matrix $A$ has
the properties that $\Pf(A)^2 = \det(A)$ and
$\Pf (P A P^T) = (\det P) \Pf(A)$.  The latter tells us that
$G \circ \wedge^2 P = (\det P) G$.  It is also not hard to show that
$\det(\wedge^2 P) = (\det P)^3$.

\begin{definition}
\label{def:action}
Two models are {\em $K$-equivalent} if they are in the same orbit for
the action of $K^\times \times \PGL_4(K)$ via
\[ (c,P): (\lambda,H) \mapsto \left(c\lambda,
    \frac{c}{\det P} H \circ \wedge^2 P\right). \]
\end{definition}
It may be checked using the above observations that this is a well
defined (right) group action on the space of models (for a fixed
choice of genus $2$ curve $C$).

\begin{example}
  \label{ex1}
  Let $C/\Q$ be the genus $2$ curve given by $y^2 = f(x)$ where
\[ f(x) = -28 x^6 + 84 x^5 - 323 x^4 + 506 x^3 - 471 x^2 + 232 x - 60. \]
  One of the elements of the 2-Selmer group of $\Jac C$ is
  represented by the model
\begin{align*}
(\lambda_1&,H_1) = (42336, \,\, 25128 z_{12}^2 + 24480 z_{12} z_{13} + 14031 z_{12} z_{23} +
    15408 z_{12} z_{14} \\ & + 13959 z_{12} z_{24} + 25407 z_{12} z_{34} + 2232 z_{13}^2 -
    16407 z_{13} z_{23} + 4464 z_{13} z_{14} \\ & - 22815 z_{13} z_{24} + 1161 z_{13} z_{34} +
    2329 z_{23}^2 + 15282 z_{23} z_{14} + 7687 z_{23} z_{24} \\ & - 19547 z_{23} z_{34} -
    2304 z_{14}^2 - 17838 z_{14} z_{24} - 22590 z_{14} z_{34} - 134 z_{24}^2 \\ & +
                    41978 z_{24} z_{34} - 99584 z_{34}^2).
\end{align*}
Applying the transformation $(c,P)$ with $c = 1/3024$ and
\begin{equation}
\label{cob}
  P = \begin{pmatrix}
    2 & -19 & 2 & 5 \\
    4 & 4 & -31 & 38 \\
    2 & 2 & 37 & 40 \\
    -7 & -7 & -14 & 7
\end{pmatrix}
\end{equation}
gives the $\Q$-equivalent model    
\begin{align*}
(\lambda_2&,H_2) = (14, \,\, z_{12} z_{23} + 2 z_{12} z_{14} - z_{12} z_{24} + 8 z_{12} z_{34} - 7 z_{13}^2 -
    13 z_{13} z_{23} \\ & - 12 z_{13} z_{14} - 15 z_{13} z_{24} - 20 z_{13} z_{34} - 5 z_{23}^2 -
    2 z_{23} z_{14} - 25 z_{23} z_{24} \\ & - 59 z_{23} z_{34} - 4 z_{14}^2 - 14 z_{14} z_{24} -
                  18 z_{14} z_{34} + 17 z_{24}^2 - 37 z_{24} z_{34} - 11 z_{34}^2).
         \end{align*}
\end{example}

\subsection{Relation to previous work}

The change of coordinates~\eqref{cob} 
was found by a combination of two techniques, called minimisation and
reduction.  {\em Minimisation} seeks to remove prime factors from a
suitably defined invariant (usually the discriminant). The prototype
example is using Tate's algorithm to compute a minimal Weierstrass
equation for an elliptic curve.  {\em Reduction} seeks to a make a
final unimodular substitution so that the coefficients are as small as
possible. The prototype example is the reduction algorithm for
positive definite binary quadratic forms.

Algorithms for minimising and reducing 2-, 3-, 4- and 5-coverings of
elliptic curves are given by Cremona, Fisher and Stoll \cite{CFS},
and Fisher \cite{F}, building on earlier work of Birch and
Swinnerton-Dyer \cite{BSD} for $2$-coverings.  Algorithms for
minimising some other representations associated to genus $1$ curves
are given by Fisher and Radicevic \cite{FR}.  A general framework for
minimising hypersurfaces is described by Kollar~\cite{K}, and this has
been refined by Elsenhans and Stoll \cite{ES}; in
particular they give practical algorithms for plane curves
(of arbitrary degree) and for cubic
surfaces. Algorithms for minimising Weierstrass equations for general
hyperelliptic curves are given by Q. Liu \cite{L}.

In this paper we give an algorithm for minimising $2$-coverings of
genus $2$ Jacobians. These are represented by pairs of quadratic forms
(see Definition~\ref{def:model}) where the first quadratic form is
fixed. We only consider minimisation and not reduction, since the
latter is already treated in~\cite[Remark 4.3]{genus2ctp}.

Our minimisation algorithm plays a key role in the work of the first
author and Jiali~Yan \cite{genus2ctp} on computing the Cassels-Tate
pairing on the $2$-Selmer group of a genus~$2$ Jacobian. Indeed the
method presented in {\em loc. cit.} for computing the Cassels-Tate
pairing relies on being able to find rational points on certain
twisted Kummer surfaces. Minimising and reducing our representatives
for the $2$-Selmer group elements simplifies the equations for these
surfaces, and so makes it more likely that we will be able to find
such rational points.

Earlier works on minimisation (see in particular \cite{CFS})
considered both minimisation theorems (i.e., general bounds on the
minimal discriminant) and minimisation algorithms (i.e., practical
methods for finding a minimal model equivalent to a given one). For
$2$-coverings of hyperelliptic Jacobians, some minimisation theorems
have already been proved; see the papers of Bhargava and
Gross~\cite[Section~8]{BG}, and Shankar and
Wang~\cite[Section~2.4]{SW}. We will not revisit these results, as our
focus is on the minimisation algorithms.

\begin{remark}
\label{rem:index2}
As noted in \cite[Lemma~17.1.1]{CF}, \cite[Section~19.1]{FH}
and \cite[Section~2.4]{genus2ctp} the
quadratic form $G = z_{12}z_{34}-z_{13}z_{24}+z_{23}z_{14}$
has two algebraic families of $3$-dimensional
isotropic subspaces. Moreover, the transformations considered in
Definition~\ref{def:action} do not describe the full projective
orthogonal group of $G$, but only the index~$2$ subgroup that
preserves (rather than swaps over) these two algebraic families.
Restricting attention to this index~$2$ subgroup (when defining
equivalence) makes no difference to the minimisation problem
(see Remark~\ref{rem:dual}), but as explained in
\cite[Sections~2.4 and~2.5]{genus2ctp} it is important in the context
of $2$-descent, since it means we can distinguish between elements of
the $2$-Selmer group with the same image in the fake $2$-Selmer group.
\end{remark}

Some Magma \cite{magma} code accompanying this article, including
an implementation of our algorithm,
will be made available from the first author's website. 

\subsection*{Acknowledgements}
This work originated as a summer project carried out by the second author
and supervised by the first author. We thank the Research in the CMS
Programme for their support.

\section{Statement of the algorithm}

We keep the notation of Section~\ref{sec:models}, but now let $K$ be a
field with discrete valuation $v : K^\times \to \Z$, valuation ring
$\OK$, uniformiser $\pi$, and residue field $k$. If $F$ is a
polynomial with coefficients in $K$ then we write $v(F)$ for the
minimum of the valuations of its coefficients.

\begin{definition} A model $(\lambda,H)$ is {\em integral} if
  $v(H) \geqslant 0$. It is {\em minimal} if $v(\lambda)$ is minimal
  among all $K$-equivalent integral models.
\end{definition}  

Using the action of $K^\times$ (see Definition~\ref{def:action})
to clear denominators it is clear that
any model is $K$-equivalent to an integral model.  By
Definition~\ref{def:model} we have
$v(\lambda) \geqslant (v(f_6) - v(f_i))/(6-i)$ for all
$i= 0,1, \ldots, 5$.  We cannot have $f_0 = \ldots = f_5 = 0$ since
$C$ is a smooth curve of genus $2$.  Therefore $v(\lambda)$
is bounded below, and minimal models exist.

It also follows from Definition~\ref{def:model}
that if $v(f_6) = v(\disc f) = 0$ then any integral model $(\lambda,H)$
has $v(\lambda) \geqslant 0$. Therefore, in global applications,
minimality is automatic at all but a finite set of primes,
which we may determine by factoring.

Returning to the local situation, there is an evident recursive
algorithm for computing minimal models if we can solve
the following problem.

\bigskip

\noindent
{\bf Minimisation problem.}
Given an integral quadratic form $H \in \OK[z_{12}, \ldots,z_{34}]$
determine whether there exists $P \in \PGL_4(K)$ such
that \[v\left(\frac{1}{\det P} H \circ \wedge^2 P\right) > 0\] and
find such a matrix $P$ if it exists.

\bigskip

Our solution to this problem (see Algorithm~\ref{main_alg}) is an
iterative procedure that computes the required transformation as a
composition of simpler transformations. These simpler transformations
are either given by a matrix in $\GL_4(\OK)$, in which case we call
the transformation an {\em integral change of coordinates}, or given
by one of the following operations, corresponding to
$P = \diag(1,1,1,\pi), \diag(1,1,\pi,\pi)$ or $\diag(1,\pi,\pi,\pi)$.

\begin{definition}
  \label{def_ops}
  We define the following three operations on
  quadratic forms $H$:
\begin{itemize}
\item{Operation~1.} Replace $H$ by
  $\frac{1}{\pi} H(z_{12},z_{13},z_{23},\pi z_{14},\pi z_{24},\pi z_{34})$,
\item{Operation~2.} Replace $H$ by
  $H(\pi^{-1} z_{12},z_{13},z_{23},z_{14},z_{24},\pi z_{34})$,
\item{Operation~3.} Replace $H$ by
  $\frac{1}{\pi} H(z_{12},z_{13},\pi z_{23},z_{14},\pi z_{24},\pi z_{34})$,
\end{itemize}
\end{definition}

The following algorithm suggests some transformations that we might try
applying to $H$. In applications $W \subset k^6$ will be a subspace
determined by the reduction of $H$ mod $\pi$.  We write
$e_{12},e_{13},e_{23},e_{14},e_{24},e_{34}$ for the standard basis of $k^6$,
and identify the dual basis with $z_{12},z_{13},z_{23},z_{14},z_{24},z_{34}$.

\begin{algorithm}
  \label{alg:totry}
  (Subalgorithm to suggest some transformations.)  We take as input an
  integral quadratic form $H \in \OK[z_{12}, \ldots,z_{34}]$ and a
  vector space $W \subset k^6$ that is isotropic for $G$. When we make
  an integral change of coordinates we apply the same transformation
  (or rather its reduction mod $\pi$) to $W$. The output is either one
  or two transformations $P \in \PGL_4(K)$.
\begin{itemize}
\item If $\dim W = 1$ then make an integral change of coordinates so
  that $W = \langle e_{12} \rangle$. Then apply Operation~2.
\item If $\dim W = 2$ then make an integral change of coordinates so
  that $W = \langle e_{12},e_{13} \rangle$. Then apply either
  Operation~1 or Operation~3.
\item If $\dim W = 3$ then either make an integral change of
  coordinates so that $W = \langle e_{12}, e_{13}, e_{23}\rangle$ and
  apply Operation~1, or make an integral change of coordinates so that
  $W = \langle e_{12}, e_{13}, e_{14}\rangle$ and apply Operation~3.
\end{itemize}
\end{algorithm}

We write $\Hbar \in k[z_{12},\ldots,z_{34}]$ for the reduction of $H$
mod $\pi$.  If $\charic(k) \not=2$ then the rank and kernel of $\Hbar$
are defined as the rank and kernel of the corresponding $6 \times 6$
symmetric matrix. If $\charic(k)=2$ then we assume that $k$ is
perfect, so that
\[ \Hbar = \frac{\partial \Hbar}{\partial z_{12}} = \ldots =
  \frac{\partial \Hbar}{\partial z_{34}} = 0 \]
defines a $k$-vector space, which we call $\ker \Hbar$.  We
then define
\[ \rank \Hbar = 6 - \dim \ker \Hbar. \]

We continue to write $G$ for the reduction of
$G$ mod $\pi$, as it should always be clear from the context which of
these we mean.

\begin{algorithm} \label{main_alg}
  (Minimisation algorithm.)  We take as input an integral quadratic
  form $H \in \OK[z_{12}, \ldots,z_{34}]$.  The output is {\tt
    TRUE/FALSE} according as whether there exists $P \in \PGL_4(K)$
  such that
  \[ v \left( \frac{1}{\det P}H \circ \wedge^2 P \right) > 0.\]
  \begin{itemize}
  \item[Step~1.] Compute $r = \rank \Hbar$. If $r=0$ then return {\tt TRUE}.
  \item[Step~2.] If $r=1$ then try making an integral change of
    coordinates so that $\Hbar = z_{34}^2$. If the reductions of $G$
    and $\pi^{-1} H(z_{12}, \ldots,z_{24},0)$ mod $\pi$ have a common
    $3$-dimensional isotropic subspace $W \subset \ker \Hbar$, then
    (since running Algorithm~\ref{alg:totry} on any such subspace $W$
    gives $v(H) > 0$) return {\tt TRUE}.
  \item[Step~3.] If $r = 2$ then try running Algorithm~\ref{alg:totry}
    on each codimension $1$ subspace $W \subset \ker \Hbar$ that is
    isotropic for $G$. If one of the suggested transformations gives
    $v(H) > 0$ then return {\tt TRUE}.
  \item[Step~4.] If $r \in \{1,2\}$ and $\Hbar$ factors as a product
    of linear forms defined over $k$, say $\Hbar = \ell_1 \ell_2$,
    then for each $i=1,2$ try making an integral change of coordinates
    so that $\ell_i=z_{34}$ and then apply Operation~2. If at least
    one of these transformations gives $v(H) \geqslant 0$ then select
    one with $\rank \Hbar$ as small as possible and go to Step~1.
  \item[Step~5.] If $r \in \{ 2,3,4,5 \}$ then try running
    Algorithm~\ref{alg:totry} on $W = \ker \Hbar$ if this subspace is
    isotropic for $G$, and otherwise on each codimension $1$ subspace
    $W \subset \ker \Hbar$ that is isotropic for $G$. If at least one
    of the suggested transformations gives $v(H) \geqslant 0$ then
    select one with $\rank \Hbar$ as small as possible and go to
    Step~1.
  \item[Step~6.] If this step is reached, or if after visiting Step~1
    the first time and returning to it a further 4 times we still do
    not have $v(H) > 0$, then return {\tt FALSE}.
\end{itemize}
\end{algorithm}

There is no difficulty in modifying the algorithm so that when it
returns {\tt TRUE} the corresponding transformation $P \in \PGL_4(K)$
is also returned.  In Section~\ref{sec:implement} we give further
details of the implementation, in particular explaining how we make
the integral changes of coordinates, and giving further details of
Step 2. In Sections~\ref{sec:weights} and~\ref{sec:proof} we prove
that Algorithm~\ref{main_alg} is correct.

\section{Remarks on implementation}
\label{sec:implement}

In Algorithms~\ref{alg:totry} and~\ref{main_alg} we are asked to try
making various integral changes of coordinates.  It is important to
realise that we are restricted to considering matrices of the form
$\wedge^2 P$ for $P \in \GL_4(\OK)$, and not general elements of
$\GL_6(\OK)$. Therefore some care is required both in determining
whether a suitable transformation exists, and in finding one when it
does.

Since the natural map $\GL_4(\OK) \to \GL_4(k)$ is surjective, we may
concentrate on the mod $\pi$ situation here. Notice however that in
the global application with $K = \Q$ and $v=v_p$ it is better to use
the surjectivity of $\SL_4(\Z) \to \SL_4(\Z/p\Z)$, so that
minimisation at $p$ does not interfere with minimisation at other
primes.

Let $k^4$ have basis $e_1, \ldots, e_4$.  We identify
$\wedge^2 k^4 = k^6$ via $e_i \wedge e_j \mapsto e_{ij}$.  Each linear
subspace $W \subset k^6$ determines a linear subspace
$V_0 \subset k^4$ given by
\begin{equation*}
V_0 =\{v \in k^4 \mid v \wedge w = 0 \text{ for all } w \in W \}
\end{equation*}
where $\wedge$ is the natural map
$k^4 \times \wedge^2 k^4 \to \wedge^3 k^4$.  Let $V_1$ be the analogue
of~$V_0$ when $W$ is replaced by its orthogonal complement with
respect to $G$.
\begin{lemma}
  \label{lem:cc}
  Let $W \subset k^6$ be a subspace, and let $P\in \GL_4(k)$.
  \begin{enumerate}
  \item If $\dim W = 1$ then $\wedge^2 P$ sends $W$ to
    $\langle e_{12} \rangle$ if and only if $P$ sends $V_0$ to
    $\langle e_1, e_2 \rangle$.
  \item If $\dim W = 2$ or $3$ then $\wedge^2 P$ sends $W$ to
    a subspace of $\langle e_{12},e_{13},e_{14} \rangle$
    if and only if $P$ sends $V_0$ to $\langle e_1 \rangle$.
  \item If $\dim W = 5$ then $\wedge^2 P$ sends $W$ to
    $\langle e_{12},e_{13},e_{14},e_{23},e_{24}\rangle$ if and only if
    $P$ sends $V_1$ to $\langle e_1, e_2 \rangle$.
  \end{enumerate}
\end{lemma}
\begin{proof}
  In (i) we have $W = \langle e_{12} \rangle$ if and only if $V_0 =
  \langle e_1, e_2 \rangle$, and in (ii) we have $W \subset \langle
  e_{12},e_{13},e_{14} \rangle$ if and only if $V_0 = \langle e_1
  \rangle$.  Since the definition of $V_0$ in terms of
  $W$ behaves well under all changes of coordinates this proves (i)
  and (ii).  As noted in Section~\ref{sec:models}, all transformations
  of the form $\wedge^2 P$ preserve
  $G$ (up to a scalar multiple). Therefore (iii) follows from (i) on
  replacing $W$ by its orthogonal complement with respect to $G$.
\end{proof}

\begin{remark}
  \label{rem:dual}
  Let $\bG$ be the matrix of second partial derivatives of $G$, i.e.,
  the $6 \times 6$ matrix with entries $1,-1,1,1,-1,1$ on the
  antidiagonal. A direct calculation shows that for any $4 \times 4$
  matrix $P$ we have
  \begin{equation*}
    \wedge^2 (\adj(P)^T) = (\det P) \bG (\wedge^2 P) \bG.
  \end{equation*}
  Letting $\PGL_4$
  act on the space of quadratic forms via
  $P : H \mapsto \frac{1}{\det P}H \circ \wedge^2 P$,
  this tells us that applying $P$
  to a quadratic form $H(z_{12},z_{13},z_{23},z_{14},z_{24},z_{34})$ has the
  same effect as applying $P^{-T}$ to its {\em dual quadratic form} 
  which we define to be $H(z_{34},-z_{24},z_{14},z_{23},-z_{13},z_{12})$.
  We note that the substitution used to replace $H$ by its dual swaps
  over the two families of isotropic subspaces
  in Remark~\ref{rem:index2}.
\end{remark}

We find the changes of coordinates in Algorithm~\ref{alg:totry} by
using Lemma~\ref{lem:cc}(i) and (ii), and the analogue of (ii) after
passing to the dual as in Remark~\ref{rem:dual}.  We find the changes
of coordinates in Steps~2 and~4 of Algorithm~\ref{main_alg} using
Lemma~\ref{lem:cc}(iii).

\begin{remark}
  In Step~2 of Algorithm~\ref{main_alg} we must find if possible a
  3-dimensional subspace
  $W \subset \langle e_{12},e_{13},e_{14},e_{23},e_{24}\rangle$ that
  is isotropic for both $G$ and $\Hbar_1$ where
  \[ H_1(z_{12}, \ldots,z_{24}) = \pi^{-1} H(z_{12}, \ldots,z_{24},0).\]
  To be isotropic for $G$ we need that
  $\langle e_{12} \rangle \subset W$.  So such a subspace $W$
  can only exist if $\overline{H}_1(1,0, \ldots,0)=0$. We assume
  that this is the case and write
  \[ \overline{H}_1(z_{12}, \ldots,z_{24}) = z_{12}
    h_1(z_{13},z_{23},z_{14},z_{24}) +
    h_2(z_{13},z_{23},z_{14},z_{24}) \] where $h_i$ is a 
  homogeneous polynomial of degree $i$.  Our problem reduces to that
  of finding a line contained in
  \[ \{ z_{13}z_{24} - z_{23}z_{14} = h_1 = h_2 = 0 \} \subset
    \PP^3. \] The well known description of the lines on
  $\{ z_{13}z_{24} - z_{23}z_{14} = 0 \} \subset \PP^3$ suggests
  that we substitute
  $(z_{13},z_{23},z_{14},z_{24}) = (x_1y_1,x_1y_2,x_2y_1,x_2y_2)$ into
  $h_1$ and $h_2$, take the GCD, and factor into irreducibles.  The
  lines of interest now correspond to linear factors of the
  form $\alpha x_1 + \beta x_2$ or $\gamma y_1 + \delta y_2$.
\end{remark}

\begin{remark}
  In Steps 3 and 5 of Algorithm~\ref{main_alg}, when $\ker \Hbar$ is
  not itself isotropic for $G$, we must find all codimension $1$
  subspaces of $\ker \Hbar$ that are isotropic for $G$. Since the
  restriction of $G$ to $\ker \Hbar$ is a non-zero quadratic form, it
  can have at most two linear factors. There are therefore at most two
  codimension $1$ subspaces we need to consider.  In particular, the
  number of times that Algorithm~\ref{main_alg} applies one of the
  operations in Definition~\ref{def_ops} is uniformly bounded.
\end{remark}

\section{Weights and Admissibility}
\label{sec:weights}

Let $H \in \OK[u_0, \ldots,u_5]$ be an integral quadratic form and
suppose that there exists $P \in \GL_4(K)$ such that
\[v\left(\frac{1}{\det P} H \circ \wedge^2 P\right) > 0.\] Then $P$
is equivalent to a matrix in Smith normal form, say
\[P = U {\rm Diag}(\pi^{w_1},\pi^{w_2},\pi^{w_3},\pi^{w_4}) V\] for
some $U,V \in \GL_4(\OK)$ and $w_1,w_2,w_3,w_4 \in \Z$.  We say that
the weight $w = (w_1,w_2,w_3,w_4)$ is {\em admissible} for $H$. It is
clear that permuting the entries of $w$, or adding the same integer to
all entries, has no effect on admissibility.

\begin{definition}
\label{def:dom}
The weight $w = (w_1,w_2,w_3,w_4)$ {\em dominates} the weight
$w' = (w'_1,w'_2,w'_3,w'_4)$ if
\begin{equation}
  \label{wt:ineq}
  \begin{aligned}
    &\max(1 + w_1 + w_2 + w_3 + w_4 - w_i - w_j - w_k - w_l,0) \\
    & \hspace{5em} \geqslant \max(1 + w'_1 + w'_2 + w'_3 + w'_4 - w'_i
    - w'_j - w'_k - w'_l,0)
  \end{aligned}
\end{equation}
for all $1 \leqslant i < j \leqslant 4$ and $1 \leqslant k < l \leqslant 4$.
\end{definition}

This definition is motivated by the fact that if $w$ dominates $w'$
and $w$ is admissible for $H$ then $w'$ is admissible for $H$. Our
next lemma shows that (for the purpose of proving that
Algorithm~\ref{main_alg} is correct) it suffices to consider finitely
many (in fact 12) weights.
  
\begin{lemma}\label{minimalcaselist}
  Every weight $w = (0,a,b,c) \in \Z^4$ with
  $0 \leqslant a \leqslant b \leqslant c$ dominates one of the
  following weights
\begin{align*}  &  (0,0,0,0), \,\,
    (0,0,0,1), \,\, (0,1,1,1), \,\,
    (0,0,1,1), \,\,
    (0,0,1,2), \,\, (0,1,2,2), \\
   & (0,1,1,2), \,\,
    (0,1,1,3), \,\, (0,2,2,3), \,\,
    (0,1,2,3), \,\,
    (0,1,2,4), \,\, (0,2,3,4).
\end{align*}
\end{lemma}

\begin{proof}
  We list the pairs $(i,j)$ and $(k,l)$ in Definition~\ref{def:dom} in
  the order $(1,2)$, $(1,3)$, $(2,3)$, $(1,4)$, $(2,4)$, $(3,4)$.
  Taking $w = (0,a,b,c)$, the left hand side of~\eqref{wt:ineq} is
  $\max(\xi,0)$ where $\xi$ runs over the entries of the following
  symmetric matrix.
  \small
  \begin{equation*}
    \hspace{-0.5em}
    \begin{bmatrix}
      1+b+c-a & 1+c & 1+c-a & 1+b & 1+b-a & 1 \\
      1+c & 1+a+c-b & 1+c-b & 1+a & 1 & 1+a-b \\
      1+c-a & 1+c-b & 1+c-a-b & 1  & 1-a & 1-b \\
      1+b & 1+a & 1 & 1+a+b-c & 1+b-c & 1+a-c \\
      1+b-a & 1 & 1-a & 1+b-c  & 1+b-a-c & 1-c \\
      1 & 1+a-b & 1-b & 1+a-c & 1-c & 1+a-b-c
    \end{bmatrix}
  \end{equation*}
  \normalsize

  We divide into $8$ cases according as to which of the inequalities
  $0 \leqslant a \leqslant b \leqslant c$ are equalities. In fact we make
  the following more precise claims.
  \begin{itemize}
  \item If $0=a=b=c$ then $w = (0,0,0,0)$.
  \item If $0=a=b<c$ then $w$ dominates $(0,0,0,1)$.
  \item If $0=a<b=c$ then $w$ dominates $(0,0,1,1)$.
  \item If $0=a<b<c$ then $w$ dominates $(0,0,1,2)$.
  \item If $0<a=b=c$ then $w$ dominates $(0,1,1,1)$.
  \item If $0<a=b<c$ then $w$ dominates $(0,1,1,3)$, $(0,1,1,2)$ or $(0,2,2,3)$.
  \item If $0<a<b=c$ then $w$ dominates $(0,1,2,2)$.
  \item If $0<a<b<c$ then $w$ dominates $(0,1,2,4)$, $(0,1,2,3)$ or $(0,2,3,4)$.
  \end{itemize}
  In each case where we list three possibilities, we further claim
  that these correspond to the subcases $a+b<c$, $a+b=c$ and $a+b>c$
  (in that order).

  Since the proofs are very similar, we give details in just one
  case. So suppose that $0<a<b<c$ and $a+b=c$. Then we have
  $a \geqslant 1$, $b\geqslant 2$, $c \geqslant 3$, $b-a \geqslant 1$,
  $c-a \geqslant 2$ and $c-b \geqslant 1$.  Listing the pairs $(i,j)$
  and $(k,l)$ in the same order as before, the left hand side
  of~\eqref{wt:ineq} is at least
  \[
    \begin{bmatrix}
      5 & 4 & 3 & 3 & 2 & 1 \\
      4 & 3 & 2 & 2 & 1 & 0 \\
      3 & 2 & 1 & 1 & 0 & 0 \\
      3 & 2 & 1 & 1 & 0 & 0 \\
      2 & 1 & 0 & 0 & 0 & 0 \\
      1 & 0 & 0 & 0 & 0 & 0
    \end{bmatrix}
  \]
  with equality if $(a,b,c)=(1,2,3)$. Therefore $w$ dominates
  $(0,1,2,3)$.
\end{proof}

Our next remark further reduces the number of weights we must consider.

\begin{remark}
  \label{rem:dual-wts}
  It is clear from Remark~\ref{rem:dual} that if $w \in \Z^4$ is
  admissible for $H$ then $-w$ is admissible for the dual of $H$.  We
  say that the weights $w$ and $-w$ (or any weights equivalent to
  these, in the sense of permuting the entries, or adding the same
  integer to all entries) are {\em dual}.  The list of $12$ weights in
  Lemma~\ref{minimalcaselist} consists of $4$ dual pairs $(0,a,b,c)$
  and $(0,c-b,c-a,c)$ with $a+b \not=c$, and $4$ self-dual weights
  $(0,a,b,a+b)$.
\end{remark}

\section{Completion of the proof}
\label{sec:proof}

In this section we complete the proof that Algorithm~\ref{main_alg} is
correct.

We first note that if $H$ and $H'$ are related by an integral change
of coordinates, and the algorithm works for $H$ then it works for
$H'$. This is because before applying Operations 1, 2 or 3 we always
make an integral change of coordinates that, by Lemma~\ref{lem:cc}, is
unique up to an element of $\GL_4(\OK)$ whose reduction mod $\pi$
preserves a suitable subspace of $k^4$.  The following elementary
lemma then shows that the transformed quadratic forms are again
related by an integral change of coordinates.

\begin{lemma}\label{keylemma}
  Let $\alpha=\diag (I_r, \pi I_{4-r})$ and $P\in \GL_4(\OK)$.  Then
  $P\in \alpha \GL_4(\OK) \alpha^{-1}$ if and only if the reduction of
  $P$ mod $\pi$ preserves the subspace
  $\langle e_1,\dots,e_r \rangle$.
\end{lemma}
\begin{proof} This is \cite[Lemma 4.1]{CFS}.
\end{proof}

Let $H \in \OK[z_{12}, \ldots, z_{34}]$ be a quadratic form.  If there
exists $P \in \PGL_4(K)$ such that
\begin{equation}
\label{aim}
v\left(\frac{1}{\det P} H \circ \wedge^2 P\right) > 0,
\end{equation}
then, as explained in Section~\ref{sec:weights}, one of the $12$
weights in Lemma~\ref{minimalcaselist} is admissible for $H$. Since
the analysis for dual weights (see Remark~\ref{rem:dual-wts}) is
essentially identical, we only need to consider one weight from each dual
pair.  It therefore suffices to consider the 8 weights listed in the
table below.

In the case of weight $(w_1, \ldots, w_4)$ we may suppose, by an
integral change of coordinates, that~\eqref{aim} holds with
$P = \diag(\pi^{w_1} ,\ldots, \pi^{w_4})$. This implies certain lower
bounds on the valuations of the coefficients of $H$.  To specify
these (in a way that is valid even when $\charic(k)=2$), we relabel the
variables $z_{12},z_{13},z_{23},z_{14},z_{24},z_{34}$ as
$z_1, \ldots, z_6$ and write
$H = \sum_{i \leqslant j} H_{ij} z_i z_j$.  We also put
$H_{ji} = H_{ij}$.  Then the lower bounds on the $v(H_{ij})$ are as
recorded in the table.

\small
\begin{center}
\begin{tabular}{|c|c|c|c|} \hline
  {\bf Case~1:} $(0,0,0,0)$ & {\bf Case~2:} $(0,0,0,1)$ &
  {\bf Case~3:} $(0,0,1,1)$ & {\bf Case~4:} $(0,1,1,2)$ \\
$\begin{bmatrix}
1 & 1 & 1 & 1 & 1 & 1 \\
1 & 1 & 1 & 1 & 1 & 1 \\
1 & 1 & 1 & 1 & 1 & 1 \\
1 & 1 & 1 & 1 & 1 & 1 \\
1 & 1 & 1 & 1 & 1 & 1 \\
1 & 1 & 1 & 1 & 1 & 1 
\end{bmatrix}$
& 
$\begin{bmatrix}
2 & 2 & 2 & 1 & 1 & 1 \\
2 & 2 & 2 & 1 & 1 & 1 \\
2 & 2 & 2 & 1 & 1 & 1 \\
1 & 1 & 1 & 0 & 0 & 0 \\
1 & 1 & 1 & 0 & 0 & 0 \\
1 & 1 & 1 & 0 & 0 & 0 
\end{bmatrix}$
& 
$ \begin{bmatrix}
3 & 2 & 2 & 2 & 2 & 1 \\
2 & 1 & 1 & 1 & 1 & 0 \\
2 & 1 & 1 & 1 & 1 & 0 \\
2 & 1 & 1 & 1 & 1 & 0 \\
2 & 1 & 1 & 1 & 1 & 0 \\
1 & 0 & 0 & 0 & 0 & 0 
\end{bmatrix}$
&
$ \begin{bmatrix}
3 & 3 & 2 & 2 & 1 & 1 \\
3 & 3 & 2 & 2 & 1 & 1 \\
2 & 2 & 1 & 1 & 0 & 0 \\
2 & 2 & 1 & 1 & 0 & 0 \\
1 & 1 & 0 & 0 & 0 & 0 \\
1 & 1 & 0 & 0 & 0 & 0 
\end{bmatrix}$ \\
  $r = 0$ & $r = 1,2,3$ & $r = 1,2$ & $r=1,2,3,4$ \\ \hline
  {\bf Case~5:} $(0,0,1,2)$ & {\bf Case~6:} $(0,1,1,3)$ &
  {\bf Case~7:} $(0,1,2,3)$ & {\bf Case~8:} $(0,1,2,4)$ \\
$  \begin{bmatrix}
4 & 3 & 3 & 2 & 2 & 1 \\
3 & 2 & 2 & 1 & 1 & 0 \\
3 & 2 & 2 & 1 & 1 & 0 \\
2 & 1 & 1 & 0 & 0 & 0 \\
2 & 1 & 1 & 0 & 0 & 0 \\
1 & 0 & 0 & 0 & 0 & 0  
\end{bmatrix}$ &
$ \begin{bmatrix}
4 & 4 & 3 & 2 & 1 & 1 \\
4 & 4 & 3 & 2 & 1 & 1 \\
3 & 3 & 2 & 1 & 0 & 0 \\
2 & 2 & 1 & 0 & 0 & 0 \\
1 & 1 & 0 & 0 & 0 & 0 \\
1 & 1 & 0 & 0 & 0 & 0 
\end{bmatrix}$ &
$ \begin{bmatrix}
5 & 4 & 3 & 3 & 2 & 1 \\
4 & 3 & 2 & 2 & 1 & 0 \\
3 & 2 & 1 & 1 & 0 & 0 \\
3 & 2 & 1 & 1 & 0 & 0 \\
2 & 1 & 0 & 0 & 0 & 0 \\
1 & 0 & 0 & 0 & 0 & 0 
\end{bmatrix} $ & 
$ \begin{bmatrix}
6 & 5 & 4 & 3 & 2 & 1 \\
5 & 4 & 3 & 2 & 1 & 0 \\
4 & 3 & 2 & 1 & 0 & 0 \\
3 & 2 & 1 & 0 & 0 & 0 \\
2 & 1 & 0 & 0 & 0 & 0 \\
1 & 0 & 0 & 0 & 0 & 0 
\end{bmatrix}$ \\
  $r = 3,4$ & $r = 3,4$ & $r = 3,4$ & $r=5$ \\ \hline      
\end{tabular}
\end{center}
\normalsize

\medskip

In our analysis of each case, we will assume we are not in an earlier
case. The possibilities for $r = \rank \Hbar$ will be justified below,
but are recorded in the table for convenience.  We complete the proof
that Algorithm~\ref{main_alg} is correct by going through the $8$
cases. In fact we show that if the cases are grouped as
\begin{center} \begin{tabular}{c|c|c|c|c}
  \multirow{2}{*}{Case 1} & Case 2 & \multirow{2}{*}{Case 4}
  & Case 5 & Case 7 \\ & Case 3 &  & Case 6 & Case 8
\end{tabular} \end{center}
then at each iteration of the algorithm we move at least one
column to the left. Therefore, if after visiting Step~1 the first
time and returning to it a further 4 times we still do not have
$v(H) > 0$ then the algorithm is correct to return {\tt FALSE}.

\subsection*{Case~1:} $w =(0,0,0,0)$. In this case we already have
$v(H) > 0$, so $r=0$ and we are done by Step~1.

\subsection*{Case~2:} $w =(0,0,0,1)$. We see from the table that
$\langle e_{12},e_{13},e_{23} \rangle \subset \ker \Hbar$ and so
$r \leqslant 3$.  We cannot have $r=0$, otherwise we would be in
Case~1. If $r=1$ then we are done by Step~2. If $r=2$ then we are done
by Step~3. If $r=3$ then Step~5 directly applies Operation~1.
(By ``directly'' we mean that there is no preliminary integral
change of coordinates.)
Since this gives $v(H) > 0$ we are in Case~1 on the next iteration.

\subsection*{Case~3:} $w =(0,0,1,1)$. We see from the table that
$\Hbar = \ell z_{34}$ for some linear form $\ell$.  One of the
transformations considered in Step~4 is to directly apply
Operation~2. Since this gives $v(H) > 0$  we are in Case~1
on the next iteration.

\subsection*{Case~4:} $w=(0,1,1,2)$. We see from the table that
$\langle e_{12},e_{13} \rangle \subset \ker \Hbar$ and so
$r \leqslant 4$.  If $r = 4$ then Step~5 directly applies Operation~1
or Operation~3.  Then on the next iteration either $(0,0,0,1)$ or
$(0,1,1,1)$ is admissible, which means we are in Case~2 or its dual.
If $r \leqslant 3$ then by applying a block diagonal element of
$\GL_4(\OK)$ with blocks of sizes $1$, $2$ and $1$, we may suppose
that $H_{35} \equiv H_{45} \equiv 0 \pmod{\pi}$.  If $r=3$ then
$\ker \Hbar = \langle e_{12},e_{13}, a e_{23} + b e_{14} \rangle$ for
some $a,b \in k$. If $a=0$ or $b=0$ then $\ker \Hbar$ is isotropic for
$G$. Otherwise $\langle e_{12},e_{13} \rangle$ is the unique
codimension~1 isotropic subspace. Either way, Step~5 directly applies
Operation~1 or Operation~3, and we are done as before.

We now suppose that $r \leqslant 2$ and divide into the following cases.
\begin{itemize}
\item Suppose that $H_{36} \not\equiv 0 \pmod{\pi}$ and
  $H_{46} \not\equiv 0 \pmod{\pi}$. Since $r \leqslant 2$ we have
  $\Hbar = \ell z_{34}$ for some linear form $\ell$.  Since there is
  no integral change of coordinates taking $\ell$ to $z_{34}$ the only
  possible outcome of Step~4 is to directly apply Operation~2. This
  brings us to Case~3.
\item Suppose that $H_{36} \equiv 0 \pmod{\pi}$ and
  $H_{46} \not\equiv 0 \pmod{\pi}$. Then $v(H_{33})=1$, otherwise we
  would be in Case~2. We again have $\Hbar = \ell z_{34}$ for some
  linear form $\ell$.  Although there does now exist an integral
  change of coordinates taking $\ell$ to $z_{34}$, following this up
  with Operation~2 does not preserve that $v(H) \geqslant 0$.  So
  again the only possible outcome of Step~4 is to directly apply
  Operation~2. This brings us to Case~3.
\item Suppose that $H_{36} \not\equiv 0 \pmod{\pi}$ and
  $H_{46} \equiv 0 \pmod{\pi}$. This is essentially the same
  as the previous case by duality.
\item Suppose that $H_{36} \equiv H_{46} \equiv 0 \pmod{\pi}$.  Then
  $\Hbar$ is a quadratic form in $z_{24}$ and $z_{34}$ only. If this
  factors over $k$ then either of the transformations in Step~4 brings
  us to Case~3. Otherwise we proceed to Step~5 which directly applies
  Operation~1 or Operation~3.  As before, this brings us to Case~2 or
  its dual.
\end{itemize}

\subsection*{Case~5:} $w=(0,0,1,2)$. Applying a block diagonal element
of $\GL_4(\OK)$ with blocks of sizes $2$, $1$ and $1$, we may suppose
that $H_{26} \equiv 0 \pmod{\pi}$. Then
$H_{36} \not\equiv 0 \pmod{\pi}$ (otherwise we would be in Case~2) and
$H_{44},H_{45},H_{55}$ cannot all vanish mod $\pi$ (otherwise we would
be in Case~3). Therefore
$\langle e_{12},e_{13} \rangle \subset \ker \Hbar$ and $r = 3$ or
$4$. The only $3$-dimensional isotropic subspaces for $G$ that contain
$\langle e_{12},e_{13} \rangle$ are
$\langle e_{12},e_{13},e_{23} \rangle$ and
$\langle e_{12},e_{13},e_{14} \rangle$. Therefore one of the
transformations considered in Step~5 is to directly apply Operation~1
or Operation~3 (the latter only being a possibility if
$H_{44} \equiv 0 \pmod{\pi}$). It follows that at the next iteration
we have $r \leqslant 2$, and so are in Case~4 or earlier.

\subsection*{Case~6:} $w=(0,1,1,3)$. Applying a block diagonal element
of $\GL_4(\OK)$ with blocks of sizes $1$, $2$ and $1$, we may suppose
that $H_{15} \equiv 0 \pmod{\pi^2}$. We have
$H_{44} \not\equiv 0 \pmod{\pi}$ (otherwise we would be in Case~4) and
$H_{35} \not\equiv 0 \pmod{\pi}$ (otherwise we would be in
Case~5). Therefore $\langle e_{12},e_{13} \rangle \subset \ker \Hbar$
and $r = 3$ or $4$. Exactly as in Case~5 we find that at the next
iteration we have $r \leqslant 2$, and so are in Case~4 or earlier.

\subsection*{Case~7:} $w=(0,1,2,3)$. We have
$H_{26} \not\equiv 0 \pmod{\pi}$ (otherwise we would be in Case~4),
and $H_{35},H_{45},H_{55}$ cannot all vanish mod $\pi$ (otherwise we
would be in Case~3). Therefore $r = 3$ or $4$, and
$\langle e_{12} \rangle \subset \ker \Hbar \subset \langle
e_{12},e_{13},e_{23},e_{14} \rangle$.

If $r = 4$ then
$\ker \Hbar = \langle e_{12}, a e_{13} + b e_{23} + c e_{14} \rangle$
for some $a,b,c \in k$.  If $b, c \not=0$ then
$\langle e_{12} \rangle$ is the unique codimension $1$ subspace of
$\ker \Hbar$ that is isotropic for $G$. Therefore, Step~5 directly
applies Operation~2, which brings us to Case~4.  If $b=0$ then
$c \not= 0$, and by applying a block diagonal element of $\GL_4(\OK)$
with blocks of sizes $1$, $1$ and $2$, we may suppose that $a =
0$. Then the $3$-dimensional isotropic subspaces for $G$ containing
$\ker \Hbar = \langle e_{12}, e_{14} \rangle$ are
$\langle e_{12}, e_{13}, e_{14} \rangle$ and
$\langle e_{12}, e_{14}, e_{24} \rangle$. Step~5 applies either
$\diag(1,\pi,\pi,\pi)$ or $\diag(1,1,\pi,1)$ bringing us to Case~5 or
Case~6. The case $c = 0$ is similar by duality.

If $r=3$ then $\ker \Hbar = \langle e_{12}, e_{23} + a e_{13},
e_{14} + b e_{13} \rangle$ for some $a,b \in k$. By applying a
block diagonal element of $\GL_4(\OK)$ with blocks of sizes $2$ and $2$,
we may suppose that $a=b=0$. 
Then
$H_{35} \equiv H_{36}\equiv H_{45}\equiv H_{46} \equiv 0 \pmod{\pi}$ and
$H_{55} \not\equiv 0 \pmod{\pi}$. The codimension $1$ subspaces of
$\ker \Hbar = \langle e_{12}, e_{23}, e_{14} \rangle$ that are isotropic
for $G$ are $\langle e_{12}, e_{23} \rangle$ and
$\langle e_{12}, e_{14} \rangle$.  The $3$-dimensional
isotropic subspaces for $G$ containing one of these spaces are
\[ \langle e_{12}, e_{13}, e_{23} \rangle, \,\,\, \langle e_{12}, e_{13},
  e_{14} \rangle, \,\,\, \langle e_{12}, e_{23}, e_{24} \rangle, \,\,\,
  \langle e_{12}, e_{14}, e_{24} \rangle. \]
The first two of these correspond to directly
applying Operation~1 or Operation~3, which brings us to Case~5 or its
dual. The last two correspond to transformations which fail to
preserve that $v(H) \geqslant 0$, and so cannot be selected by Step~5.

\subsection*{Case~8:} $w=(0,1,2,4)$. We have
$H_{35} \not\equiv 0 \pmod{\pi}$ (otherwise we would be in Case~5),
$H_{26} \not\equiv 0 \pmod{\pi}$ (otherwise we would be in Case~6),
and $H_{44} \not\equiv 0 \pmod{\pi}$ (otherwise we would be in
Case~7). Therefore $r=5$ and $\ker \Hbar = \langle e_{12}
\rangle$. Step~5 directly applies Operation~2 which brings us to
Case~6.

\medskip

\begin{example}
  We give three examples where Algorithm~\ref{main_alg} takes the
  maximum of $4$ iterations to give $v(H) > 0$. The first two examples
  start in Case~7, with $\rank \Hbar = 3$ or $4$, and the final one
  starts in Case~8. In the first two examples there are two choices on
  the first iteration. We made an arbitrary choice in each case,
  but in fact with the other choices the algorithm would still
  have taken $4$ iterations.

Let $K = \Q$ and $v = v_p$ for any choice of prime number $p$.
An arrow labelled $(w_1, \ldots, w_4)$ indicates that we replace
$H$ by $\frac{1}{\det P} H \circ \wedge^2 P$ where
$P = \diag(p^{w_1}, \ldots, p^{w_4})$.
\begin{align*}
  p^5 z_{12}^2 + z_{13} z_{34} + p z_{23}^2 + p z_{14}^2 + z_{24}^2
  & \stackrel{(0,0,0,1)}{\ra} 
    p^4 z_{12}^2 + z_{13} z_{34} + z_{23}^2 + p^2 z_{14}^2 + p z_{24}^2 \\
  & \stackrel{(0,0,1,0)}{\ra}
    p^3 z_{12}^2 + p z_{13} z_{34} + p z_{23}^2 + p z_{14}^2 + z_{24}^2 \\
  & \stackrel{(0,1,0,1)}{\ra}
    p^3 z_{12}^2 + z_{13} z_{34} + p z_{23}^2 + p z_{14}^2 + p^2 z_{24}^2 \\
  & \stackrel{(0,0,1,1)}{\ra}
  p( z_{12}^2 + z_{13} z_{34} + z_{23}^2 + z_{14}^2 + p z_{24}^2).
\end{align*}
\begin{align*}
  p^5 z_{12}^2 + z_{13} z_{34} + p z_{23}^2 + z_{14} z_{24}
  & \stackrel{(0,0,0,1)}{\ra} 
    p^4 z_{12}^2 + z_{13} z_{34} + z_{23}^2 + p z_{14} z_{24} \\
  & \stackrel{(0,0,1,0)}{\ra}
    p^3 z_{12}^2 + p z_{13} z_{34} + p z_{23}^2 + z_{14} z_{24} \\
  & \stackrel{(0,1,0,1)}{\ra}
    p^3 z_{12}^2 + z_{13} z_{34} + p z_{23}^2 + p z_{14} z_{24} \\
  & \stackrel{(0,0,1,1)}{\ra}
  p( z_{12}^2 + z_{13} z_{34} + z_{23}^2 + z_{14} z_{24}). 
\end{align*}
\begin{align*}
  p^6 z_{12}^2 + z_{13} z_{34} + z_{23} z_{24} + z_{14}^2
  & \stackrel{(0,0,1,1)}{\ra} 
    p^4 z_{12}^2 + p z_{13} z_{34} + z_{23} z_{24} + z_{14}^2 \\
  & \stackrel{(0,0,0,1)}{\ra}
    p^3 z_{12}^2 + p z_{13} z_{34} + z_{23} z_{24} + p z_{14}^2 \\
  & \stackrel{(0,1,0,1)}{\ra}
    p^3 z_{12}^2 + z_{13} z_{34} + p z_{23} z_{24} + p z_{14}^2 \\
  & \stackrel{(0,0,1,1)}{\ra}
  p( z_{12}^2 + z_{13} z_{34} + z_{23} z_{24} + z_{14}^2). 
\end{align*}
\end{example}

\bibliographystyle{alpha}

\begin{thebibliography}{9}

\bibitem[BG]{BG}
  M.~Bhargava and B.H.~Gross, The average size of the 2-Selmer group
  of Jacobians of hyperelliptic curves having a rational Weierstrass
  point, in {\em Automorphic representations and $L$-functions},
  D. Prasad, C.~S. Rajan, A.~Sankaranarayanan and J.~Sengupta (eds.),
  23--91, Tata Institute of Fundamental Research, Stud. Math. {\bf 22},
  Mumbai, 2013.

\bibitem[BSD]{BSD}
  B.J.~Birch and H.P.F.~Swinnerton-Dyer,     
  Notes on elliptic curves. I,
  {\em J. reine angew. Math.} {\bf 212} (1963), 7--25.

\bibitem[BCP]{magma}
W. Bosma, J. Cannon and C. Playoust, 
The Magma algebra system I: The user language, {\em J. Symb. Comb.} {\bf{24}}, 
235-265 (1997), \url{http://magma.maths.usyd.edu.au/magma/} 

\bibitem[CF]{CF}
  J.W.S.~Cassels and E.V.~Flynn,
  {\em Prolegomena to a middlebrow arithmetic of curves of genus 2},
  London Mathematical Society Lecture Note Series, {\bf 230},
  Cambridge University Press, Cambridge, 1996.
  
\bibitem[CFS]{CFS} 
  J.E. Cremona, T.A. Fisher, and M. Stoll,
  Minimisation and reduction of 2-,3- and 4-coverings of elliptic curves,
  {\em Algebra \& Number Theory} {\bf 4} (2010), no. 6, 763--820.

\bibitem[ES]{ES} A.-S. Elsenhans and M. Stoll, {\em Minimization of
    hypersurfaces}, preprint, 2021, \\
  \url{https://arxiv.org/abs/2110.04625}
  
\bibitem[F]{F}
  T.A. Fisher,  
  Minimisation and reduction of 5-coverings of elliptic curves,
  {\em Algebra \& Number Theory} {\bf 7} (2013), no. 5, 1179--1205.

\bibitem[FR]{FR}
  T.A. Fisher and L. Radičević,     
  Some minimisation algorithms in arithmetic invariant theory,
  {\em J. Théor. Nombres Bordeaux} {\bf 30} (2018), no. 3, 801--828.

\bibitem[FY]{genus2ctp}
  T.A. Fisher and J. Yan, {\em Computing the
    Cassels-Tate pairing on the 2-Selmer group of a genus 2 Jacobian},
  preprint, 2023, \url{https://arxiv.org/abs/2306.06011}
  
\bibitem[FH]{FH}  
  W. Fulton and J. Harris, {\em Representation theory. A first course},
  Graduate Texts in Mathematics {\bf 129}, Springer-Verlag, New York, 1991.

\bibitem[K]{K}
  J. Kollár,
  Polynomials with integral coefficients, equivalent to a given polynomial,
  {\em Electron. Res. Announc. Amer. Math. Soc.} {\bf 3} (1997), 17--27.

\bibitem[L]{L}
  Q. Liu, {\em Computing minimal Weierstrass equations}, preprint, 2022, \\
  \url{https://arxiv.org/abs/2209.00469}

\bibitem[SW]{SW} A. Shankar and X. Wang,
  Rational points on hyperelliptic curves having a marked non-Weierstrass point,
  {\em Compos. Math.} 154 (2018), no. 1, 188--222.

\end{thebibliography}

\end{document}